\documentclass[10pt]{article}
\usepackage{amssymb,amsfonts}
\usepackage{pb-diagram}
\usepackage{amsmath,amsthm,amsfonts,amssymb}
\usepackage[usenames]{color}
\usepackage{hyperref}
\usepackage{soul}
\usepackage{graphicx}
\usepackage{amssymb}
\usepackage{amsmath}
\usepackage{amssymb}

\newtheorem{theorem}{Theorem}
\newtheorem{definition}{Definition}
\newtheorem{example}{Example}
\newtheorem{lemma}{Lemma}
\newtheorem{corollary}{Corollary}
\newtheorem{remark}{Remark}

\newtheorem{proposition}{Proposition}

\newcommand{\be}{\begin{enumerate}}
\newcommand{\ee}{\end{enumerate}}
\newcommand{\beq}{\begin{equation}}
\newcommand{\eeq}{\end{equation}}
\newcommand{\ov}[1]{\mbox{${\overline{#1}}$}}

\def\bF{{\mathbb F}}

\def\bZ{{\mathbb{Z}}}

\newcommand{\RT}{\rT} 
\DeclareMathOperator{\rT}{T}

\newcommand{\tr}{{\rm tr}} 

\title{Drinfeld-Manin solutions of the Yang-Baxter equation coming from cube complexes}
\author{Alina Vdovina}
\begin{document}

\maketitle

\begin{abstract}
The most common geometric interpretation of the Yang-Baxter equation is by braids, knots and relevant Reidemeister moves. So far,  cubes were used for connections with
the third Reidemeister move only. We will show that there are higher-dimensional cube complexes solving the $D$-state Yang-Baxter equation for arbitrarily large $D$.
More precisely, we introduce explicit constructions of cube complexes covered by products of $n$ trees and show that these cube complexes lead to new solutions of  the Yang-Baxter equations.

\end{abstract}

\section{Introduction}

The Yang-Baxter equations appear in many fields of mathematics and theoretical physics:
they are defined in many ways, possess many interpretations and occur in many contexts.

Often in the literature, the Yang-Baxter equation is called ``triangular'' ---  see, for example \cite{frenkel} --- but we will show that a very natural framework lies in those $n$-dimensional cube complexes 
whose universal cover is a Cartesian product of $n$ trees.
Our main result is that each $n$-dimensional cube complex whose universal cover is a Cartesian product of $n$ trees 
induces a Drinfeld-Manin solution of the Yang-Baxter equation.
The following definition distills the result of analysing over 30 years' worth of papers.

\begin{definition}
Let $X$ be a (non-empty) set, and $R \colon X^2 \to X^2$ be a bijection given by
$R(x, y) = (u, v)$.
We call $R$ a {\em Drinfeld-Manin solution} of the Yang-Baxter equation if
$$R^{12}R^{23}R^{12} = R^{23}R^{12}R^{23}$$
on $X^3$, where $R^{ij}$ means $R$ acting on $i$-th and $j$-th component of $X^3$.
\end{definition}

This definition is quite close to what is called a {\em set-theoretical solution} or a {\em Yang-Baxter map} in the literature; see, \cite{etingof, veselov, drinfeld}. 
See also \cite{agata} for the most recent account of the subject.

\begin{theorem}
Let $G$ be a group acting simply transitively on a product of $n \geq 3$ trees $\Delta$, and let $P$ be a cube complex obtained as the quotient of $\Delta$ by the action of $G$.
Then $G$ induces a Drinfeld-Manin solution $R$ of the Yang-Baxter equation in the following way:
edges of $P$ inherit labelings and orientations from the action; 
the set $X$ is taken to be the set of labels on the edges of the cubes; 
the bijection $R$ is induced by the 2-cells of $P$ --- more precisely, if $x_ix_jx_kx_l$ is a label 
of a square, then $R(x_i,x_j)=(x_l^{-1},x_k^{-1})$ whereas 
if $x_i$ and $x_j$ do not appear next to each other in a 2-cell of $C$ then $R(x_i,x_j)=(x_i,x_j)$.
\end{theorem}

The group $G$ is a natural invariant of the Drinfeld-Manin solution; we shall  call it the {\it cube complex group}.
This distinguishes results of the present paper from earlier work.

\begin{corollary}
Let $P$ be the flip map  $P \colon  X^2 \to X^2$, such that $P(x,y)=(y,x)$. 
Then  the map $Q=P \circ R$ is a solution to the quantum Yang-Baxter equation $$R^{12}R^{13}R^{23} = R^{23}R^{13}R^{12}, $$ where $R$
is the map from the Theorem~1. 

\end{corollary}

The connection between  the Yang-Baxter equation and the quantum Yang-Baxter equation was established in \cite{etingof}.

\begin{definition}
 Two solutions $(X, r)$ and $(X^{\prime}, r^{\prime})$  are
isomorphic if there exists a bijective map $\nu  \colon X\to X^{\prime}$
such that $r^{\prime}(\nu(x), \nu(y)) = (\nu(\sigma x(y)), \nu (\gamma y(x)))$,
where $r(x, y) = (\sigma x(y), \gamma y(x))$, for $x, y \in X$.
\end{definition}

\begin{proposition}
The Drinfeld-Manin solutions associated to non-isometric cube complexes are not isomorphic.
\end{proposition}

\begin{corollary}
Formula $(4)$ of Section 3.3 gives a lower bound to the number of Drinfeld-Manin solutions with $\left| X \right|=2(m+l+k)$, where $m,l,k \geq 2$.
\end{corollary}

\section{The connection with the quantum Yang-Baxter equation}
If we consider a vector space $V$ spanned by elements of $X$, we get a solution of the quantum $D$-state Yang-Baxter equation, where $D=\left|X \right|$; see, \cite{H, drinfeld, etingof}. 
In most applications to theoretical physics, $V$  is a finite-dimensional space over $\mathbb{C}$.
The idea of introducing extra structure on a basis of $V$ goes back to \cite{H}, 
where the basis vectors were indexed by elements of a  finite field or ring.
We will show that one may use even more elaborated algebraic and arithmetic structures: namely, algebras coming from Hamiltonian quaternions and more general 
quaternion algebras of finite characteristic. Also, the novelty of our approach lies in the combination of the arithmetic and algebra with the geometry of $CAT(0)$ cube complexes
and buildings. Details on how to introduce the structure of a Bruhat-Tits building on a CAT(0) cube complex can be found in \cite{RSV}. 

In \cite{KaufmanYB}, it was shown that a unitary solution of  the Yang-Baxter equation plays the role of a universal quantum gate. 
It was additionally stated there that it is surprisingly difficult to produce such solutions.
However, every Drinfeld-Manin solution generates unitary solutions of the Yang-Baxter equation and, since we are able to produce several infinite families of such solutions, it would be interesting to investigate possible applications in the spirit of \cite{KaufmanYB}.

In order to understand the connection between knot theory and solvable models, Jones
proposed a procedure he called {\em baxterization} to produce solutions of the Yang-Baxter
equation from representations of the braid group, \cite{jonesBaxterization}, which created a whole area connecting quantum physics, representation theory and theory of C*-algebras:
namely Temperley-Lieb algebras; see for example, \cite{jaffeliuwozn} and references in the paper. 
The connection with the braid group and Temperley-Lieb algebras puts some additional restrictions on solutions of the Yang-Baxter equation.
The solutions obtained by ``complexification'' of Drinfeld-Manin solutions  do not have these restrictions, so we plan to investigate if further generalizations of universal quantum gates are possible \cite{KaufmanYB}. In the case of CAT(0) complexes covered by product of trees, we may produce another class of C*-algebras which are closely related to higher-dimensional generalizations of Thompson groups and can be viewed as higher-dimensional analogues of Cuntz-Krieger algebras \cite{LV, V}.

We would like to mention another potential application here. Possible connections of quantum computation and geometric group theory/combinatorics of words were indicated in \cite{Manin}, but these connections were not fully developed, since quantum mechanics requires explicit constructions of 
geometric objects of dimensions three and higher, and these objects have been constructed only recently \cite{RSV}.    

\section{Connections between new and old invariants}

The standard invariants of solutions of the Yang-Baxter equation are structure groups and structure semigroups (see, for example, \cite{etingof})
which can be defined for any mappings $X^2 \to X^2$ for any set $X$.

\begin{definition}Let $X$ be a set and $S \colon  X^2 \to X^2$ a mapping. 
The {\em structure group} $G(X)$ is the group generated by elements of $X$ with defining relations
 $xy = tz$ when $S(x, y) = (t, z)$.
\end{definition}

\begin{definition}Let $X$ be a set and $S \colon  X^2 \to X^2$ a mapping. 
The {\em structure semigroup} $G'(X)$ is the semigroup generated by elements of $X$ with defining relations
 $xy = tz$ when $S(x, y) = (t, z)$.
\end{definition}

The structure semigroups of the Drinfeld-Manin solutions also posses the structure of higher-rank graphs with one vertex considered, for example, in \cite{LV}. 
There is a connection made between the Yang-Baxter equation and higher-rank graphs in \cite{dilian}, but concrete examples involve a very particular class of solutions of the Yang-Baxter equation, different from ours (one way to show this is to construct the geometric group and its homology groups).

Let $R,P$ be  from Theorem 1. We observe that there is an involution $\delta$ on the set $X$ induced by the directions of edges.
Then the relations of the structure semigroup naturally fall into four equivalence classes corresponding to geometric squares of the complex  $P$.
It is easy to see that the complex cube group is a quotient of the structure group. 

\section{Groups acting simply transitively on products of trees}

\subsection{Complexes covered by products of trees}

We first introduce square complexes covered by products of two trees and then describe how cube complexes covered by products of $n$ trees can be obtained from
$n$ square complexes.
The general reference for $CAT(0)$ cube complexes is \cite{sageev}.

A {\em square complex} $S$ is a $2$-dimensional combinatorial cell complex: its $1$-skeleton consists of a graph $S^1 = (V(S),E(S))$ with set of vertices $V(S)$, and set of oriented edges $E(S)$. The $2$-cells of the square complex come from a set of squares $S(S)$ that are combinatorially glued to the graph $S^1$ as explained below. Reversing the orientation of an edge $e \in E(S)$ is written as $e \mapsto e^{-1}$ and the set of unoriented edges is the quotient set 
\[
\ov{E}(S) = E(S)/(e \sim e^{-1}).
\]

More precisely, a square $\square$ is described by an equivalence class of $4$-tuples of oriented edges $e_i \in E(S)$
\[
\square = (e_1,e_2,e_3,e_4)
\]
where  the origin of $e_{i+1}$ is the terminus of $e_i$ (with $i$ modulo $4$). Such $4$-tuples describe the same square if and only if they belong to the same orbit under the dihedral action  generated by cyclically permuting the edges $e_i$ and by  the reverse orientation  map 
\[
(e_1,e_2,e_3,e_4) \mapsto (e_4^{-1},e_3^{-1},e_2^{-1},e_1^{-1}).
\]
We think of a square-shaped $2$-cell glued to the (topological realization of the) respective edges of the graph $S^1$.
For more details on square complexes, we refer the reader to, for example,  \cite{burger-mozes:lattices}. 
Examples of square complexes are given by products of trees.

\begin{remark}
We note that in our definition of a square complex, each square is determined by its boundary.
The group actions considered in the present paper are related only to such complexes.
\end{remark}

Let $T_l$ denote the $l$-valent tree. The product of trees 
\[
M = T_{m} \times T_{l}
\]
is a Euclidean building of rank $2$ and a square complex.  Here we are interested in  lattices; that is, those groups $\Gamma$ acting discretely and cocompactly on $M$ respecting the structure of square complexes.  The quotient $S = \Gamma \backslash M$ is a finite square complex, typically with an orbifold structure coming from the stabilizers of cells.

We are especially interested in the case where $\Gamma$ is torsion free and acts simply transitively on the set of vertices of $M$. These yield the smallest quotients $S$ without non-trivial orbifold structure. Since $M$ is a CAT(0) space, any finite group stabilizes a cell of $M$ by the Bruhat--Tits fixed point lemma. Moreover, the stabilizer of a cell is profinite,
hence compact, so that a discrete group $\Gamma$ acts with trivial stabilizers on $M$ if and only if $\Gamma$ is torsion free.

Let $S$ be a square complex. For $x \in V(S)$, let $E(x)$ denote the set of oriented edges originating in $x$. 
The {\em link} at the vertex $x$ in $S$ is the (undirected multi-)graph ${\mathbb Lk}_x$ whose set of vertices is $E(x)$ and whose set of edges in ${\mathbb Lk}_x$ joining vertices $a,b \in E(x)$ are given by corners $\gamma$ of squares in $S$ containing the respective edges of $S$,  see \cite{burger-mozes:lattices}. 

A covering space of a square complex admits a natural structure as a square complex such that the covering map preserves the combinatorial structure. In this way, a connected square complex admits a universal covering space.

\begin{proposition}
The universal cover of a finite connected square complex is a product of trees if and only if the link $ {\mathbb Lk}_x$ at each vertex $x$ is a complete bipartite graph.
\end{proposition}
\begin{proof}
This is well known and follows, for example, from \cite[Theorem~C]{Ballmann-Brin1995}.
\end{proof}

\subsection{VH-structure}  \label{sec:VHstructure}

A  {\em vertical/horizontal structure}, in short a {\em VH-structure}, on a square complex $S$ consists of a bipartite structure $\ov{E}(S) = E_v \sqcup E_h$ on the set of unoriented edges of $S$ such that for every vertex $x \in S$ the link ${\mathbb Lk}_x$ at $x$ is the complete bipartite graph on the induced partition of $E(x) = E(x)_v \sqcup E(x)_h$. Edges in $E_v$ (resp.\ in $E_h$) are called vertical (resp.\ horizontal) edges. See \cite{wise1} for general facts on VH-structures. The {\em partition size} of the VH-structure is the function  $V(S) \to 
\mathbb N \times \mathbb N$ on the set of vertices 
\[
x \mapsto (\# E(x)_v, \# E(x)_h)
\]
or just the corresponding tuple of integers if the function is constant. Here $\#(-)$ denotes the cardinality of a finite set. 

We record the following basic fact; see \cite{burger-mozes:lattices} after their Proposition 1.1.

\begin{proposition} \label{prop:VHstructureanduniversalcover}
Let $S$ be a square complex. Then the following are equivalent:
\begin{enumerate}
\item 
The universal cover of $S$ is a product of trees $T_m \times T_l$ and the group of covering transformations does not interchange the factors.
\item 
There is a VH-structure on $S$ of constant partition size $(m,l)$. \hfill $\square$
\end{enumerate}
\end{proposition}

\begin{corollary} \label{cor:latticesversusVH}
Torsion free cocompact lattices $\Gamma$ in  $\mathbb Aut(T_m) \times \mathbb Aut(T_l)$, not interchanging the factors and 
up to conjugation, correspond uniquely to finite square complexes with a VH-structure of partition size $(m,l)$ up to isomorphism.
\end{corollary}
\begin{proof}
A lattice $\Gamma$ yields a finite square complex $S = \Gamma \backslash T_m \times T_l$ of the desired type. Conversely, a finite square complex $S$ with VH-structure of constant partition size $(m,l)$ has universal covering space $M = T_m \times T_l$ by Proposition~\ref{prop:VHstructureanduniversalcover}, and the choice of a base point vertex $\tilde{x} \in M$ above the vertex $x \in S$  identifies $\pi_1(S,x)$ with the lattice $\Gamma = \mathbb Aut(M/S) \subseteq \mathbb Aut(T_m) \times \mathbb Aut(T_l)$. The lattice depends on the  chosen base points only up to conjugation.
\end{proof}

Simply transitive torsion free lattices not interchanging the factors as in Corollary~\ref{cor:latticesversusVH} correspond to square complexes with only one vertex and a VH-structure, necessarily of constant partition size. These will be studied in the next section.

\subsection{$1$-vertex square complexes}

Let $S$ be a square complex with just one vertex $x \in S$ and a VH-structure $\ov{E}(S) = E_v \sqcup E_h$. Passing from the origin to the terminus of an oriented edge induces a fixed point free involution on $E(x)_v$ and on $E(x)_h$. Therefore the partition size is necessarily a tuple of even integers.

\begin{definition} \label{defi:BMstructureingroup}
A  {\em vertical/horizontal structure}, in short {\em VH-structure},  in a group $G$ is an ordered pair $(A,B)$ of finite subsets $A,B \subseteq G$ such that the following holds.
\begin{enumerate}
\item \label{defiitem:AandBinvolution} Taking inverses induces fixed point free involutions on $A$ and $B$.
\item The union $A \cup B$ generates $G$.
\item \label{defiitem:ABequalsBA} The product sets $AB$ and $BA$ have size $\#A \cdot \#B$ and $AB = BA$.
\item \label{defiitem:AB2torsion} The sets $AB$ and $BA$ do not contain $2$-torsion.
\end{enumerate}
The tuple $(\#A,\#B)$ is called the {\em valency vector} of the VH-structure in $G$.
\end{definition}


Similar to what is described in \cite[Section 6.1]{burger-mozes:lattices},
starting from a VH-structure we have the following construction 
\begin{equation} \label{eq:constructionSAB}
(A,B) \leadsto S_{A,B}
\end{equation}
yields a square complex $S_{A,B}$ with one vertex and VH-structure starting from a VH-structure $(A,B)$ in a group $G$. The vertex set $V(S_{A,B})$ contains just one vertex $x$. The set of oriented edges of $S_{A,B}$ is the disjoint union 
\[
E(S_{A,B}) = A \sqcup B
\]
with the orientation reversion map given by $e \mapsto e^{-1}$. Since $A$ and $B$ are preserved under taking inverses, there is a natural vertical/horizontal structure such that $E(x)_h = A$ and $E(x)_v = B$. 
The squares of $S_{A,B}$ are constructed as follows. Every relation in $G$ 
\begin{equation}\label{eq:relationabba} 
ab = b'a'
\end{equation}
with $a,a' \in A$ and $b,b' \in B$ (not necessarily distinct) gives rise to a $4$-tuple 
\[
(a,b,a'^{-1},b'^{-1}).
\] 
The following relations are equivalent to \eqref{eq:relationabba} 
\begin{eqnarray*}
a'b^{-1} & = &  b'^{-1}a, \\
a^{-1}b' & = & ba'^{-1}, \\
a'^{-1}b'^{-1} & = & b^{-1}a^{-1}.
\end{eqnarray*}
and we consider the four $4$-tuples so obtained
\[
(a,b,a'^{-1},b'^{-1}), \quad (a',b^{-1},a^{-1},b'), \quad (a^{-1},b',a',b^{-1}), \quad (a'^{-1},b'^{-1},a,b)
\]
as equivalent. A square $\square$ of $S_{A,B}$ consists of an equivalence class of such $4$-tuples arising from a relation of the form \eqref{eq:relationabba}.

\begin{lemma}
The link ${\mathbb Lk}_x$ of $S_{A,B}$ in $x$ is the complete bipartite graph $L_{A,B}$ with vertical vertices labelled by $A$ and horizontal vertices labelled by $B$.
\end{lemma}
\begin{proof}
By \ref{defiitem:ABequalsBA} of Definition~\ref{defi:BMstructureingroup} every pair $(a,b) \in A \times B$ occurs on the left hand side  in a relation of the form \eqref{eq:relationabba} and therefore the link ${\mathbb Lk}_x$ contains $L_{A,B}$.

If \eqref{eq:relationabba} holds, then the set of left hand sides of equivalent relations 
\[
\{ab, a'b^{-1},a^{-1}b',a'^{-1}b'^{-1}\}
\]
is a set of cardinality $4$, because $A$ and $B$ and $AB$ do not contain $2$-torsion by Definition~\ref{defi:BMstructureingroup} \ref{defiitem:AandBinvolution} + \ref{defiitem:AB2torsion} and the right hand sides of the equations are unique by 
Definition~\ref{defi:BMstructureingroup} \ref{defiitem:ABequalsBA}. Therefore $S_{A,B}$ only contains $(\#A \cdot \#B)/4$ squares. It follows that ${\mathbb Lk}_x$ has at most as many edges as $L_{A,B}$, and, since it contains $L_{A,B}$, must agree with it.
\end{proof}

\begin{definition}
We will call the complex $S_{A,B}$ a $(\#A,\#B)$-complex, keeping in mind, that there are many complexes with the same valency vector. Also, if a group is a fundamental group of 
a $(\#A,\#B)$-complex, we will call it a $(\#A,\#B)$-group.
\end{definition}

\begin{example}
Figure 1 shows an example of a $(4,4)$-group.
\end{example}

\begin{figure}{Fig.1}

\includegraphics[height=2.7cm]{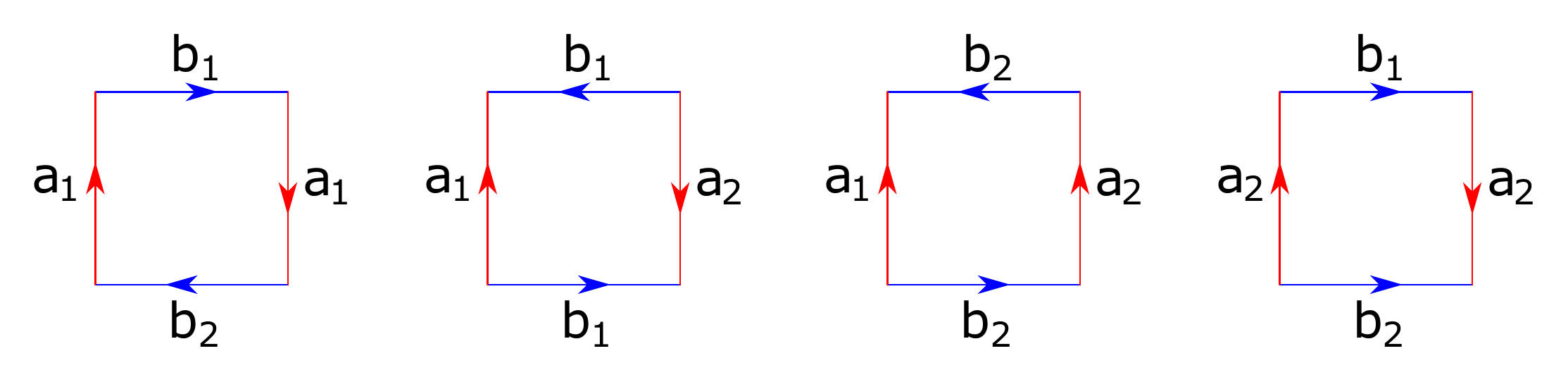}

\end{figure}

\newpage
\section{$n$-Cube groups}

\subsection{Definition of $n$-Cube groups}
We generalise VH-structure to the $n$-dimensional case.

\begin{definition} \label{defi:cubestructure}
An  {\em $n$-cube structure}  in a group $G$ is an ordered tuple $(A_1, \ldots, A_n)$ of finite subsets $A_i \subseteq G$ such that the following hold:
\begin{enumerate}
\item \label{defiitem:AandBinvolution} Taking inverses induces fixed point free involutions on $A_i$.
\item The union $\cup A_i$ generates $G$.
\item \label{defiitem:ABequalsBA} The product sets $A_iA_j$ and $A_jA_i$ have size $\#A_i \cdot \#A_j$ and $A_iA_i = A_jA_i$.
\item \label{defiitem:AB2torsion} The sets $A_iA_j$ and $A_jA_i$ do not contain $2$-torsion.
\item Group $G$ acts simply transitively on Cartesian product of $n$ trees.
\end{enumerate}
The tuple $(\#A_1, \ldots,\#A_j)$ is called the {\em valency vector} of the $n$cube-structure in $G$, and sets $(A_1, \ldots, A_n)$ are called generating sets.
\end{definition}

We note that each pair $A_i,A_j\subseteq G$ forms a subgroup $M$ of $G$ equipped with VH-structure. 

\subsection{Proof of Theorem 1}

Without loss of generality, we may consider groups with $3$-cube structures.
Let $G$ be a group with $3$-cube structure with generating sets $A,B,C \subseteq G$. Since $G$ acts simply transitively on a Cartesian product
of three trees $\Delta$, the set $X=A\cup B\cup C$ labels the edges of the quotient cube complex $P$. We show that the map $R$ from the formulation
of the Theorem 1 satisfies the Yang-Baxter  equation, namely to show, that for any $(x,y,z) \in X$,
$R^{12}R^{23}R^{12}(x,y,z) = R^{23}R^{12}R^{23}(x,y,z)$. Without loss of generality, we have to consider three cases, 
\begin{enumerate}
\item $x\in A, y\in B, z \in C$,
\item $x\in A, y\in A, z \in C$, 
\item $x\in A, y\in A, z \in A$. 
\end{enumerate}
We start with the first case. Consider the cube $K$ in $P$ such that $xyz$ is the path between the most distanced vertices in 
$K$ (path $a_1b_1c_2$ on fig. 2, for example). Then each of $R^{12}R^{23}R^{12}$ and $R^{23}R^{12}R^{23}$ will move $xyz$ to $x'y'z'$, where $x'y'z'$
is the most distanced path  parallel to $xyz$ in $K$, (path $c_4b_2^{-1}a_1^{-1}$  on fig.2). Second case follows from the fact that $A,C\subseteq G$ forms a subgroup $M$ of $G$ equipped with VH-structure. The third case is true, since $R$ is the identity map at $A$ by definition.

\begin{figure}{Fig. 2}

\includegraphics[height=4cm]{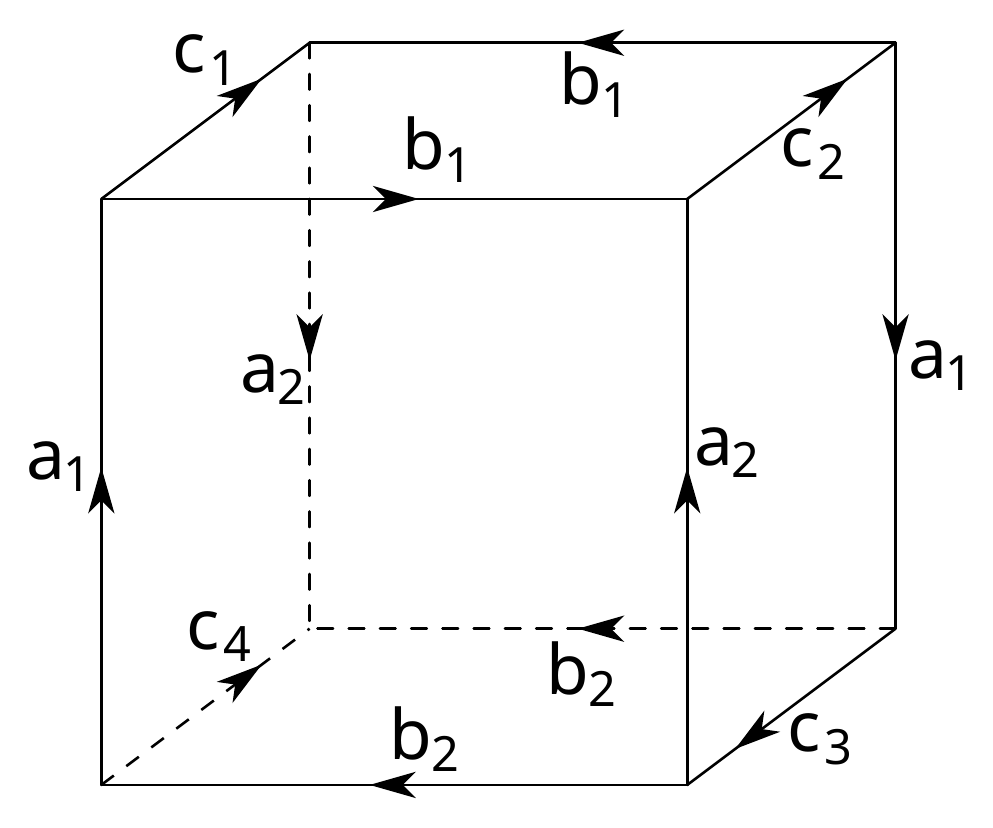}

\end{figure}

\subsection{Mass formula for complexes with $3$-cube structure}  \label{sec:mass_formula}

In this section we show that the number of non-isomorphic complexes with $3$-cube structure grows at least factorially.
Every one vertex square complex with VH-structure $(A,B)$ leads to a complex with $3$-cube structure.
Let $(A,B)$ be a VH-structure in a group $T$. We introduce a set $C$ and add to the relations of $T$ all commutators of $C$ with $A\cup B$.
Then the resulting group $G$ will be equipped with a $3$-cube structure $(A,B,C)$.
So, the mass formula for the number of one vertex square complexes with VH-structure up to isomorphism applies for the case of cube complexes.

For completeness, we present a mass formula for the number of one vertex square complexes with VH-structure up to isomorphism where each square complex is counted with the inverse order of its group of automorphisms as its weight \cite{stix-av}.

Let $A$ (resp.\ $B$, \ $C$) be a set with fixed point free involution of size $2m$ (resp.\ $2l$, \ $2k$ ). 
In order to count one vertex square complexes $S$ with VH-structure with vertical/horizontal partition $A \sqcup B$ of oriented edges  we introduce the generic matrix
\[
X = (x_{ab})_{a\in A, b \in B}
\]
with rows indexed by $A$ and columns indexed by  $B$ and with $(a,b)$-entry a formal variable $x_{ab}$. Let $X^t$ be the transpose of $X$, let $\tau_A$ (resp.\ $\tau_B$) be the permutation matrix for $e \mapsto e^{-1}$ for $A$ (resp.\ $B$). For a square $\square$  we set 
\[
x_\square = \prod_{e \in \square} x_e
\]
where the product ranges over the edges $e = (a,b)$ in the link of $S$ originating from $\square$ and $x_e = x_{ab}$. 
Then the sum of the $x_\square$, when $\square$ ranges over all possible squares with edges from $A \sqcup B$, reads 
\[
\sum_{\square} x_\square =  \frac{1}{4}\tr\big((\tau_A X \tau_B X^t\big)^2),
\]
and the number of  one vertex square complexes $S$ with VH-structure of partition size $(2m,2n)$ and edges labelled by $A$ and $B$  is given by 
\begin{equation} \label{eq:labeled_mass_formula}
\widetilde{F}_{m,l} = \frac{1}{(ml)!} \cdot \frac{\partial^{4ml}}{\prod_{a,b} \partial x_{ab} } \left( \frac{1}{4}\tr\big((\tau_A X \tau_B X^t\big)^2)\right)^{ml}. 
\end{equation}
Note that this is a constant polynomial.

We can turn this into a mass formula for the number of one vertex square complexes with VH-structure up to isomorphism where each square complex is counted with the inverse order of its group of automorphisms as its weight. We simply need to divide by the order of the universal relabelling 
\[
\#({\mathbb Aut}(A,\tau_A) \times {\mathbb Aut}(B,\tau_B)) = 2^l(l)! \cdot 2^m(m)!.
\]
Hence the mass of one vertex square complexes with VH-structure is given by
\begin{equation} \label{eq:mass_formula}
F_{m,l} = \frac{1}{2^{l+m+2lm}(l)! \cdot (m)! \cdot (ml)!} \cdot \frac{\partial^{4ml}}{\prod_{a,b} \partial x_{ab} } \left( \tr\big((\tau_A X \tau_B X^t\big)^2)\right)^{ml}.
\end{equation}
The formula \eqref{eq:labeled_mass_formula} reproduces the numerical values of $\widetilde{F}_{m,l}$ for small values $(2m,2l)$ that were computed by 
Rattaggi in \cite[Table B.3]{rattaggi:thesis}. 
Here small means $ml \leq 10$.

\subsection{Explicit Drinfeld-Manin solutions}
Several infinite series of groups acting simply-transitively on products of $n$ trees were constructed in \cite{RSV} so by Theorem 1 every such group
gives a Drinfeld-Manin solution. We adopt results from \cite{RSV} to give an explicit construction for a map $R$ satisfying Definition 1.
Let $q$ be a prime power. 
We describe the presentation of $\Lambda_S$ in terms of finite fields only.  Let 
\[
\delta \in \bF_{q^2}^\times
\]
be a generator of the multiplicative group of the field with $q^2$ elements. If $i,j \in \bZ/(q^2-1)\bZ$ are 
\[
i \not\equiv j \pmod{q-1},
\]
then $1+\delta^{j-i} \not= 0$, since otherwise
\[
1 = (-1)^{q+1} = \delta^{(j-i)(q+1)} \not= 1,
\]
a contradiction. Then there is a unique $x_{i,j} \in \bZ/(q^2-1)\bZ$ with 
\[
\delta^{x_{i,j}} = 1  + \delta^{j-i}.
\]
With these $x_{i,j}$ we set $y_{i,j} := x_{i,j} + i - j$, so that 
\[
\delta^{y_{i,j}} = \delta^{x_{i,j} + i - j} = (1  + \delta^{j-i}) \cdot \delta^{i-j} = 1  + \delta^{i-j}.
\]
We moreover set  
\begin{align*}
l(i,j) & := i -  x_{i,j}(q-1), \\
k(i,j) & := j -  y_{i,j}(q-1).
\end{align*}

If $\alpha = \delta^i$ and $\beta = \delta^j$, then 
\begin{equation}
\label{eq:concreterelation1}
\delta^{k(i,j)}  = \delta^{j -  y_{i,j}(q-1)} = \delta^j (1+ \delta^{i-j})^{1-q} =  \frac{\delta^i + \delta^j}{(1+ \delta^{i-j})^q} =  \frac{\delta^i + \delta^j}{(\delta^i + \delta^{j})^q} \cdot \delta^{jq} = \zeta_{\alpha}(\beta) \beta,
\end{equation}
and
\begin{equation}
\label{eq:concreterelation2}
\delta^{l(i,j)}  = \delta^{i-  x_{i,j}(q-1)} = \delta^i (1+ \delta^{j-i})^{1-q} =  \frac{\delta^i + \delta^j}{(1+ \delta^{j-i})^q} =  \frac{\delta^i + \delta^j}{(\delta^i + \delta^{j})^q} \cdot \delta^{iq} = \zeta_{\beta}(\alpha) \alpha.
\end{equation}

Let now $M \subseteq \bZ/(q^2-1)\bZ$ be a union of cosets under $(q-1)\bZ/(q^2-1)$, we denote the number of cosets by $n$.

It was shown in \cite{RSV} that the following groups act simply transitively on product of $n$ trees.
\[
\Gamma_{M,\delta} = \left\langle a_i  \text{ for all } i \in M \ \left| 
\begin{array}{c}
a_{i+ (q^2 - 1)/2}  a_i = 1 \text{ for all $i \in M$}, \\
a_i a_j = a_{k(i,j)}a_{l(i,j)} \text{ for all $i,j \in M$ with $i \not\equiv j \pmod{q-1}$}
\end{array}
\right.\right\rangle
\]
if $q$ is odd, and if $q$ is even:
\[
\Gamma_{M,\delta} = \left\langle a_i  \text{ for all } i \in M \ \left| 
\begin{array}{c}
a_i^2 = 1 \text{ for all $i \in M$}, \\
a_i a_j = a_{k(i,j)}a_{l(i,j)} \text{ for all $i,j \in M$ with $i \not\equiv j \pmod{q-1}$}
\end{array}
\right.\right\rangle .
\]

By Theorem 1, the map $R$ is the following:
$R(a_i, a_j) = (a_{k(i,j)}a_{l(i,j)})$  for all $i,j \in M$ with $i \not\equiv j \pmod{q-1}$ and $R(a_i,a_j)=Id $ otherwise.

The structure semigroup of this solution is 
\[
\Gamma^{\prime}_{M,\delta} = \left\langle a_i  \text{ for all } i \in M \ \left| 
\begin{array}{c}
a_i a_j = a_{k(i,j)}a_{l(i,j)} \text{ for all $i,j \in M$ with $i \not\equiv j \pmod{q-1}$}
\end{array}
\right.\right\rangle
\]

\begin{example}
\label{ex:666}
We compute the smallest example in dimension $3$ given by $q=5$ and
\[
M = \{i \in \bZ/24\bZ \ ; \ 4 \nmid i\}.
\]
The group $\Gamma_1$ acts vertex transitively on product of three trees
\[
\RT_{6} \times  \RT_{6} \times  \RT_{6},
\]

\[
\Gamma_{1} = \left\langle \begin{array}{c}
a_1,a_5,a_9,a_{13},a_{17},a_{21}, \\
b_2,b_6,b_{10},b_{14},b_{18},b_{22}, \\
c_3,c_7,c_{11},c_{15},c_{19},c_{23}
\end{array}
\ \left| 
\begin{array}{c}
a_ia_{i+12} = b_i b_{i+12} = c_ic_{i+12} = 1  \ \text{ for all $i$ }, \\

a_{1}b_{2}a_{17}b_{22}, \ 
a_{1}b_{6}a_{9}b_{10}, \ 
a_{1}b_{10}a_{9}b_{6}, \ 
a_{1}b_{14}a_{21}b_{14}, \ 
a_{1}b_{18}a_{5}b_{18}, \\ 
a_{1}b_{22}a_{17}b_{2}, \ 
a_{5}b_{2}a_{21}b_{6}, \ 
a_{5}b_{6}a_{21}b_{2}, \ 
a_{5}b_{22}a_{9}b_{22}, \\

a_{1}c_{3}a_{17}c_{3}, \ 
a_{1}c_{7}a_{13}c_{19}, \ 
a_{1}c_{11}a_{9}c_{11}, \ 
a_{1}c_{15}a_{1}c_{23}, \ 
a_{5}c_{3}a_{5}c_{19}, \\
a_{5}c_{7}a_{21}c_{7}, \ 
a_{5}c_{11}a_{17}c_{23}, \ 
a_{9}c_{3}a_{21}c_{15}, \ 
a_{9}c_{7}a_{9}c_{23}, \\

b_{2}c_{3}b_{18}c_{23}, \ 
b_{2}c_{7}b_{10}c_{11}, \ 
b_{2}c_{11}b_{10}c_{7}, \ 
b_{2}c_{15}b_{22}c_{15}, \ 
b_{2}c_{19}b_{6}c_{19}, \\
b_{2}c_{23}b_{18}c_{3}, \ 
b_{6}c_{3}b_{22}c_{7}, \ 
b_{6}c_{7}b_{22}c_{3}, \ 
b_{6}c_{23}b_{10}c_{23}.
\end{array}
\right.\right\rangle .
\]
The group $\Gamma_1$ induces a Drinfeld-Manin solution $R_1$ which can be represented by a $324\times324$ matrix.

\end{example}

For any set of size $k$ of distinct odd primes, $p_1, \ldots, p_k$, there is a group acting simply transitively on a product of $k$ trees of valencies $p_1+1,\ldots, p_k+1$, obtained using 
Hamiltonian quaternion algebras; see \cite{RSV}. 

\begin{example}

For $p_1=3,p_2=5,p_3=7$ we construct a group $G_2$ acting simply transitively on a product of three trees
 \[
\RT_{4} \times  \RT_{6} \times  \RT_{8},
\]

We set
\begin{align*}
a_1 = 1 + j + k, \ a_2 = 1+j-k, \ a_3 = 1-j-k, \ a_4 = 1 -j + k, \\
b_1 = 1 + 2i, \ b_2  = 1 + 2j, \ b_3 = 1 + 2k, \ b_4 = 1 - 2i, \ b_5 = 1- 2j, \ b_6 = 1 - 2k, \\
c_1 = 1+2i + j + k, \ c_2 = 1-2i + j + k, \ c_3 = 1+2i - j + k, \ c_ 4= 1+2i + j - k, \\ 
 c_5 = 1- 2i -  j - k, \ c_6 = 1+2i - j - k, \ c_7 = 1-2i + j - k, \ c_8 = 1-2i - j + k.
\end{align*}

With this notation we have $a_i^{-1} = a_{i+2}$, $b_i^{-1} = b_{i+3}$, and $c_i^{-1} = c_{i+4}$, and using these abbreviations  the following explicit presentation
of a group acting simply transitively and a product of three trees
was found in \cite{RSV}.

\[
\Gamma_2 = \left\langle
\begin{array}{c}
a_1,a_2 \\ 
b_1,b_2,b_3 \\
c_1,c_2,c_3,c_4
\end{array}
\ \left| 
\begin{array}{c}
a_1b_1a_4b_2,  \ a_1b_2a_4b_4, \  a_1b_3a_2b_1, \ 
a_1b_4a_2b_3,  \ a_1b_5a_1b_6, \ a_2b_2a_2b_6 \\

a_1c_1a_2c_8, \ a_1c_2a_4c_4, \ a_1c_3a_2c_2, \ a_1c_4a_3c_3, \\
a_1c_5a_1c_6, \ a_1c_7a_4c_1, \ a_2c_1a_4c_6, \ a_2c_4a_2c_7 \\

b_1c_1b_5c_4, \
b_1c_2b_1c_5, \
b_1c_3b_6c_1, \
b_1c_4b_3c_6, \
b_1c_6b_2c_3, \
b_1c_7b_1c_8, \\
b_2c_1b_3c_2, \
b_2c_2b_5c_5, \
b_2c_4b_5c_3, \
b_2c_7b_6c_4, \
b_3c_1b_6c_6, \
b_3c_4b_6c_3
\end{array}
\right.\right\rangle.
\]

The group $\Gamma_2$ induces a Drinfeld-Manin solution $R_2$ which can be represented by a $324\times324$ matrix.
\end{example}

To show that $R_1$ and $R_2$ are not isomorphic, we compute first homologies of $\Gamma_1$ and $\Gamma_2$. We use MAGMA, but this can be done by direct calculations as well. Since $H_1(\Gamma_1)=\bZ/2\bZ \times \bZ/10\bZ \times \bZ/10\bZ$, $H_1(\Gamma_2)=\bZ/2\bZ \times \bZ/2\bZ \times \bZ/4\bZ \times \bZ/4\bZ$, then, by Proposition 1,
$R_1$ and $R_2$ are not isomorphic.


\end{document}